\documentclass[12pt,sumlimits,intlimits]{amsart}
\usepackage{url,cite,hyperref,amsmath,amsthm,amssymb,mathtools}
\usepackage[margin=1.25in]{geometry}

\newtheorem{theorem}{Theorem}[section]

\newtheorem{lemma}[theorem]{Lemma}

\newcommand{\V}{\Vert}

\newcommand{\Z}{\mathbb{Z}}
\newcommand{\Q}{\mathbb{Q}}
\newcommand{\C}{\mathbb{C}}
\newcommand{\N}{\mathbb{N}}

\newcommand{\R}{\mathbb{R}}
\newcommand{\T}{\mathbb{T}}

\newcommand{\ve}{\varepsilon}

\numberwithin{equation}{subsection}

\title{A note on strongly quasidiagonal groups}
\author{Caleb Eckhardt}
\address{Department of Mathematics, Miami University, Oxford, OH, 45056}
\email{eckharc@miamioh.edu}
\thanks{Partially supported by NSF grant  DMS-1262106.}

\date{}
\begin{document}
\maketitle
\begin{abstract} In this note we address a question of Don Hadwin: ``Which groups have strongly quasidiagonal C*-algebras?"  In recent work we showed that all finitely generated virtually nilpotent groups have strongly quasidiagonal C*-algebras, while together with Carri\'on and Dadarlat we showed that most wreath products fail to have strongly quasidiagonal C*-algebras.  These two results raised the question of whether or not strong quasidiagonality could characterize virtual nilpotence among finitely generated groups.
The purpose of this note is to provide examples of finitely generated groups (in fact of the form $\Z^3\rtimes \Z^2$) that are not virtually nilpotent yet have strongly quasidiagonal C*-algebras.  Moreover we show these examples are the ``simplest" possible  by proving that a group of the form $\Z^d\rtimes \Z$ is virtually nilpotent if and only if its group C*-algebra is strongly quasidiagonal.

\end{abstract}
\section{Introduction}  This paper continues our study of strongly quasidiagonal groups from \cite{Eckhardt13} and we refer the reader to the introduction of \cite{Eckhardt13} for more motivation and the definitions used here. Although the questions we consider are also interesting for general groups, in this work we restrict our attention to discrete groups.

This investigation grew out of a question asked by Don Hadwin in \cite{Hadwin87}: ``Which groups have strongly quasidiagonal C*-algebras." 
 A few pages after Hadwin's question, Jonathon Rosenberg showed in an appendix that if $G$ is not amenable, then $C^*(G)$ is not strongly quasidiagonal.   
On the amenable side of things, we showed in \cite{Eckhardt13} that $C^*(G)$ is strongly quasidiagonal if $G$ is virtually nilpotent.  On the other hand, together with Carri\'on and Dadarlat \cite{Carrion13}, we showed that many (amenable) wreath products are \emph{not} strongly quasidiagonal.  These results raised the question of whether or not for \emph{finitely generated} groups (see \cite[Corollary 1.7]{Eckhardt13} and following discussion) strong quasidiagonality of $C^*(G)$ could characterize virtual nilpotence (or perhaps supramenability) of $G.$

In this paper we show that strong quasidiagonality cannot characterize virtual nilpotence -- even among the polycyclic groups. In particular, we give examples of groups of the form $\Z^m\rtimes \Z^n$ where $m\geq3$ and $n\geq2$ that are not virtually nilpotent but are strongly quasidiagonal. This result relies heavily on Daniel Berend's investigation of commutaive endomorphic actions on tori \cite{Berend83}.

On the other hand we show that groups of the form $\Z^3\rtimes \Z^2$ are the ``simplest" possible, non-virtually nilpotent, polycyclic, strongly quasidiagonal groups by proving that groups of the form $\Z^d\rtimes\Z$ are strongly quasidiaognal if and only if they are virtually nilpotent. It is interesting to note that in the case of $G=\Z^d\rtimes\Z$ the question of strong quasidiagonality can be reduced to a family of problems of diophantine approximation where, roughly speaking, ``good" approximations correspond to the strong quasidiagonality of $G$ while ``bad" approximations correspond to $G$ not being strongly quasidiagonal.

\section{Construction of Examples}

All groups considered in this paper are of the form $\Z^m\rtimes\Z^n.$  We then realize the group C*-algebra of $\Z^m\rtimes\Z^n$ as a C*-crossed product via  $C^*(\Z^m\rtimes\Z^n)\cong C^*(\Z^m)\rtimes \Z^n\cong C(\widehat{\Z}^m)\rtimes \Z^n.$ Archbold and Spielberg showed \cite{Archbold94} that when analyzing the ideal structure of crossed products, topological freeness of the action is a desirable property.  Let us recall  that a group $G$ of homeomorphisms of a locally compact space $X$ is \emph{topologically free} if for any finite subset $F\subseteq G\setminus\{ e \}$, the set 
\begin{equation*}
\bigcap_{\alpha\in F} \{ x\in X: \alpha(x)\neq x \}
\end{equation*}
is dense in $X.$  Since $\widehat{\Z}^m$ is connected,  all of our examples will be topologically free.
\begin{lemma} \label{lem:topfree} Let $G$ be a connected topological group and $H$ a group of continuous automorphisms of $G.$  Then the action of $H$  on $G$ is topologically free.
\end{lemma}
\begin{proof} Let $\alpha\in H$, and define the subgroup $F(\alpha)=\{ x\in G:\alpha(x)=x  \}\leq G.$ If $F(\alpha)^c$ is not dense, then $F(\alpha)$ contains a non-empty open subset, hence $F(\alpha)$ contains an open neighborhood of the identity.  Since $G$ is connected, any open neighborhood of the identity generates $G$, hence $F(\alpha)=G$ or $\alpha= id.$ Topological freeness then follows from the fact that $F(\alpha)^c$ is open.

\end{proof}
\begin{lemma} \label{lem:goodorb} Let $G=\Z^n\rtimes \Z^m$ for some $m,n\geq1$.  Suppose that for every multiplicative character $\omega\in \widehat{\Z}^n$ we have $Orb(\omega)=\{ h(\omega):h\in \Z^m \}$ is either finite or dense.  Then $C^*(G)$ is strongly quasidiagonal.
\end{lemma}
\begin{proof} Let $\pi$ be a unitary representation of $G$ that is faithful on $C^*(\Z^n).$
By Lemma \ref{lem:topfree}, the induced action of $\Z^m$ on $\widehat{\Z}^n$ is topologically free.  By \cite[Theorem 1]{Archbold94} and the fact that $\Z^m$ is amenable, it follows that $\pi$ is faithful on $C^*(G).$  Since $\widehat{\Z}^n$ is connected and $\pi$ is faithful, the spectrum of $\pi(f)$ is connected for every $f\in C(\widehat{\Z}^n).$  In particular, $\pi(C(\widehat{\Z}^n))$ contains no non-trivial compact operators.  Again by \cite[Theorem 1]{Archbold94} it follows that $\pi(C^*(G))$ contains no compact operators.

Since $\pi$ is faithful and $\pi(C^*(G))$ has trivial intersection with the compacts, by Voiculescu's theorem \cite{Voiculescu76} (see also \cite{Davidson96}) $\pi$ is approximately unitarily equivalent to the left regular representation of $G.$  Since $G$ is polycyclic, it is residually finite (see \cite{Segal83}).  By a well-known result of Bekka \cite{Bekka90}, since $G$ is amenable the group C*-algebra $C^*(G)$ is residually finite dimensional, thus the left regular representation is quasidiagonal, forcing $\pi$ to be quasidiagonal as well.

Suppose now that $\pi$ is a unitary representation that is not faithful on $C^*(\Z^n).$  Then $\textup{ker}(\pi)\cap C^*(\Z^n)$ corresponds to a closed $\Z^m$-invariant subset of $\widehat{\Z}^n.$  By assumption this set must be finite (otherwise $\pi$ would be faithful on $C^*(\Z^n)$). In this case one sees that $\pi(C^*(G))$ is a subhomogeneous C*-algebra.  Since all irreducible representations of a subhomogeneous C*-algebras are finite dimensional, it is clear that all subhomogeneous C*-algebras are strongly quasidiagonal and in particular that $\pi$ is a quasidiagonal representation.
\end{proof}
It is not obvious that there exist \emph{any} pairs $\Z^n,\Z^m$ with action satisfying the above lemma. Thankfully, Daniel Berend (building on earlier work of Furstenberg \cite{Furstenberg67})  classified all endomorphic commutative semigroup actions on tori that satisfy the  orbit condition in Lemma \ref{lem:goodorb} which we will use to build a specific example.   Let us recall
\begin{theorem}[Berend \textup{\cite[Theorem 2.1]{Berend83}}] \label{thm:Berend} Let $\Sigma$ be a commutative semigroup of endomorphisms of $\T^n.$ Every orbit of the action of $\Sigma$ on $\T^n$ is either dense or finite if and only if all three of the following three conditions are satisfied:
\begin{enumerate}
\item There is a $\sigma\in\Sigma$ such that the characteristic polynomial of $\sigma^n$ is irreducible over $\Z$ for all $n\geq1.$
\item For every common eigenvector $v$ of $\Sigma$ there is a $\sigma\in \Sigma$ such that the corresponding eigenvalue $\lambda$ of $\sigma$ has $|\lambda|>1.$
\item There are $\sigma_1,\sigma_2\in \Sigma$ such that $\sigma_1^n=\sigma_2^m$ for $m,n\in\Z$ implies that $m=n=0.$
\end{enumerate}
\end{theorem}
For the convenience of the reader we construct an explicit example of a non virtually nilpotent group satisfying Berend's conditions.
\begin{theorem} There is an action $\varphi$ of $\Z^2$ on $\Z^3$ such that the semidirect product  $\Z^3\rtimes_\varphi \Z^2$ is strongly quasidiagonal and not virtually nilpotent.
\end{theorem}
\begin{proof}
Our example is derived from elementary algebraic number theory and the results and terminology used can be found in any text on the subject, for example \cite{Ribenboim01}.

Let $\alpha$ be a root of $p(t)=t^3+t^2-2t-1$ (or any monic cubic with integer coefficients and distinct irrational roots). Let $\Q(\alpha)$ be the field generated by $\Q$ and $\alpha$, and let $A\subseteq \Q(\alpha)$ be the ring of algebraic integers.  Since $[\Q(\alpha):\Q]=3$, we have $A$ (as an additive abelian group) is isomorphic to $\Z^3$ (see \cite[Theorem 6.2.J.2]{Ribenboim01}).   Let $U$ be the multiplicative groups of units of $A.$
By Dirichlet's Units Theorem (see \cite[Theorem 10.4.J.1]{Ribenboim01}), $U$ is  finitely generated and the torsion free part of $U$ has rank 2. We consider $U\leq GL(3,\Z)$ by letting elements of $U$ act on $A\cong \Z^3$ by multiplication. 

Since $\alpha(\alpha^2+\alpha-2)=1$, we have  $\alpha\in U.$  Moreover since $1,\alpha,\alpha^2$ are linearly independent over $\Q$, by considering the action of $\alpha$ on $\Q-\textup{span}\{ 1,\alpha,\alpha^2  \}$ we see that the characteristic polynomial of $\alpha\in GL(3,\Z)$ is $p(t).$ From this it easily follows that $\alpha$ satisfies condition (1) of Theorem \ref{thm:Berend} and that for each eigenvector of $\alpha$ either $\alpha$ or $\alpha^{-1}$ will satisfy condition (2).  Since $\alpha$ has infinite order and the torsion free part of $U$ has rank 2, it follows that there is a $\beta\in U$ such that $\alpha$ and $\beta$ generate a copy of $\Z^2.$ 

Letting $\Sigma\cong \Z^2$ be the group generated by $\alpha$ and $\beta$ we see that $\Sigma$ satisfies all the hypotheses of Theorem \ref{thm:Berend}. Let $\varphi:\Z^2\rightarrow Aut(\Z^3)$ be the action of $U$ on $A$ by multiplication, restricted to $\Sigma.$ Since $\alpha$ has an eigenvalue with modulus not equal to one, it follows that $\Z^3 \rtimes_\varphi \Z^2$ is not virtually nilpotent (in fact it contains a free subsemigroup on 2 generators \cite{Rosenblatt74}).  By Lemma \ref{lem:goodorb}, $\Z^3 \rtimes_\varphi \Z^2$ is strongly quasidiagonal.
\end{proof} 
\section{The need for two automorphisms}
The question of whether or not a C*-algebra of the form $C(X)\rtimes \Z$ is strongly quasidiagonal can be reduced to a question about the relationship between the backwards and forward orbits of points in $X$ (see \cite[Theorem 25]{Hadwin87} and \cite{Smucker82,Pimsner83} for some precursors and related variations).  In the specific case of $C^*(\Z^d\rtimes \Z)\cong C(\widehat{\Z}^d)\rtimes \Z$, the question about orbits reduces to a family of diophantine approximation questions. In this section we will see that  ``good" approximations correspond to strongly quasidiagonal groups and ``bad" approximations correspond to non-strongly quasidiagonal groups.
 
Let $x\in\R.$  We write $\V x \V=\textup{dist}(x,\Z).$
A general question in the theory of diophantine approximation concerns the distribution of sequences of the form $\V t_n\xi\V.$  The specific instance of this question addressing our concerns asks,  (Q) ``Given a sequence of real numbers $t_n$, when does there exist a $\xi\in\R$ and $\ve>0$ such that $\V t_n\xi \V\geq\ve$ for all $n\geq1$?" In relation to strong quasidiagonality of $\Z^d\rtimes\Z$, two types of sequences appear: (i) (finite families of) real valued polynomials evaluated at $n=1,2,...$ and (ii) sequences of exponential growth.  In case (i) the answer to (Q) is always no by the following theorem of Cook
\begin{theorem}[Cook, \textup{\cite[Theorem]{Cook74}}] \label{thm:Cook}Let $d,R\geq1$ be integers.  Let $p_1,...,p_R$ be polynomials with real coefficients, no constant term and  degree bounded by $d.$  Then there are constants $C,\epsilon>0$ depending only on $d$ and $R$ such that for every integer $N$, there is an integer $0\leq n\leq N$ such that
\begin{equation*}
\max_{i=1,...,n}\{  \V p_i(n) \V   \}<\frac{C}{N^\epsilon}.
\end{equation*}
\end{theorem}
Using Cook's result one can prove strong quasidiagonality of virtually nilpotent groups of the form $\Z^d\rtimes \Z$ (or just use \cite{Eckhardt13}). 

In case (ii) the answer to (Q) is always yes. Independently, Khintchine \cite{Khintchine26}, Pollington \cite{Pollington79} and de Mathan \cite{Mathan80} showed that for any lacunary sequence ($(t_n)$ is lacunary if there is an $r>1$ so $t_{n+1}/t_n\geq r$ for all $n\geq1$) the answer to (Q) is yes (in fact, they proved stronger statements).

Recently, the preprint \cite{Haynes13} appeared with a simplified  proof of (a weaker version of) the Khintchine-Pollington-de Mathan result. A reading of their proof reveals that it can be easily adapted to prove the following variant:
\begin{lemma} \label{lem:expgrowth} Let $t_n$ be a sequence of real numbers and $r,C>1$ such that $C^{-1}r^n\leq t_n \leq Cr^n$ for all $n\geq1.$  Then there is a  $\xi\in\R$ and an $\epsilon>0$ such that $\V t_n\xi \V\geq\ve$ for all $n\geq1.$
\end{lemma}  
Using the above Lemma we can produce non quasidiagonal representations of non virtually nilpotent groups of the form $\Z^d\rtimes\Z.$

Recall that a unital C*-algebra $A$ is \emph{finite} if $x^*x=1$ implies $xx^*=1$ for all $x\in A.$
\begin{theorem} \label{thm:vnilsqd} Let $G=\Z^d\rtimes_\alpha \Z.$  The following are equivalent
\begin{enumerate}
\item $G$ is virtually nilpotent.
\item $C^*(G)$ is strongly quasidiagonal.
\item Every quotient of $C^*(G)$ is finite.
\end{enumerate}
\end{theorem}
\begin{proof} It is well-known that (2)  implies (3) for any C*-algebra (see for example \cite{Brown04}). It was shown in \cite{Eckhardt13} that  (1) implies (2). For the illumination of the relationship between ``good" approximations and strong quasidiagonality of groups of the form $\Z^d\rtimes\Z$ we now outline an alternate--and more direct--proof that (1) implies (2).

To this end suppose that $G$ is virtually nilpotent.  Then there is some $n\in\N$ so 1 is the only eigenvalue of $\alpha^n.$  Note that if $\alpha^n$ satisfies the hypotheses of \cite[Theorem 25]{Hadwin87} then so does $\alpha.$  Therefore assume that 1 is the only eigenvalue of $\alpha$, from which it follows that $(1-\alpha)$ is a nilpotent matrix.  By basic facts about linear algebra and free abelian groups we may assume that $\alpha$ is upper triangular with 1's along the diagonal. From this it follows that there are polynomials $p_{ij}$ for $1\leq i<j\leq d$ such that $p_{ij}(0)=0$ and
\begin{equation*}
(\alpha^n)_{ij}=\left\{ \begin{array}{ll} 0 & \textrm{ if }i>j\\
                                                          1 & \textrm{ if }i=j\\
                                                          p_{ij}(n) & \textrm{ if }i<j\\
                                 \end{array} \right.
\end{equation*}
for all $n\geq 0.$

Let $\theta\in \T^d\cong \widehat{\Z}^d$ and choose real numbers $\theta_1,...,\theta_d$  so  $\theta=(\exp(2\pi i\theta_j))_{j=1}^d.$  Consider the real polynomials 
\begin{equation*}
q_i(x)=\sum_{j=i+1}^d p_{ij}(x)\theta_j\quad \textrm{ for }i=1,...,d-1.
\end{equation*}
Let $\alpha$ also denote the induced action on $\widehat{\Z}^d.$ One sees that for $n\geq0$ we have
\begin{align*}
\alpha^n(\theta) &=\theta\cdot (\exp(2\pi i(q_1(n))),...,\exp(2\pi i(q_{d-1}(n))),1))\\
&= \theta\cdot(\exp(2\pi i(\V q_1(n))\V),...,\exp(2\pi i(\V q_{d-1}(n))\V),1)).
\end{align*}
By Theorem \ref{thm:Cook} we can find $n$ large enough so 
\newline
$(\exp(2\pi i(\V q_1(n))\V),...,\exp(2\pi i(\V q_{d-1}(n))\V),1))$ is as close to $(1,...,1)$ as we like.  Hence we can choose $n$ large enough so $\alpha^n(\theta)$ is as close to $\theta=\alpha^0(\theta)$ as we like.  Hence $G$ is strongly quasidiagonal by \cite[Theorem 25]{Hadwin87}.

In order to prove (3) implies (1) suppose  that $G$ is not virtually nilpotent.  By \cite{Wolf68} (see also \cite{Tits81}), it follows that $\alpha$ has an eigenvalue $\lambda$ with modulus greater than 1.  
For a vector $w\in\C^d$ we write $w_i\in \C$ for the $i$-th entry of $w.$
Suppose first that $\lambda\in\R.$ Without loss of generality, let $w\in\R^d$ be an eigenvector with $w_1=1.$  By \cite{Pollington79} there is an $\ve>0$ and a  $\xi\in \R$ such that $\V \xi \lambda^n \V\geq \ve$ for all $n\geq0.$  Since $\lambda>1$ we have
\begin{equation*}
\lim_{n\rightarrow\infty} \alpha^{-n}(\xi w)_1\rightarrow 0,
\end{equation*}
while 
\begin{equation*}
\alpha^n(\xi w)_1 \textup{ mod } 1 \in [\ve,1-\ve] \quad \textup{ for all }\quad n\geq1.
\end{equation*}
It now easily follows from \cite[Theorem 25]{Hadwin87} (see also \cite{Smucker82}) that $G$ has a quotient that is not finite.

Suppose now that $\lambda\in\C\setminus\R.$  We mention that if the modulus of $\lambda$ is sufficiently large, then we can apply \cite[Theorem 1.2(b)]{Dubickas10} and take real parts to complete the proof as in the real case.  For the general case, let $w\in \C^d$ be an eigenvector for $\alpha$ associated with $\lambda.$  Since the entries of $\alpha$ are real, it follows that $\overline{\lambda}$ is an eigenvalue for $\alpha$ with eigenvector $\overline{w}$.  Without loss of generality suppose that $w_1=1$ and $w_2=a+bi$ with $b\neq0.$

Set $v=1/2(w+\overline{w}).$ We have
\begin{equation*}
\lim_{n\rightarrow\infty}\alpha^{-n}(v)=\mathbf{0}\in \R^d,
\end{equation*}
and for every $n\geq 1$
\begin{equation*}
\alpha^n(v)_1=\textup{Re}(\lambda^n), \quad \alpha^n(v)_2=a\textup{Re}(\lambda^n)-b\textup{Im}(\lambda^n).
\end{equation*}
Since $|\textup{Re}(\lambda^n)|^2+|\textup{Im}(\lambda^n)|^2=|\lambda|^{2n}$, it follows that the sequence
\begin{equation*}
t_n=\max \{ |\textup{Re}(\lambda^n)|, |a\textup{Re}(\lambda^n)-b\textup{Im}(\lambda^n)| \}
\end{equation*}
satisfies the hypotheses of Lemma \ref{lem:expgrowth}.  Let $\xi,\ve$ be as in Lemma \ref{lem:expgrowth}.  It then follows that for $n\in\N$ we have either $\alpha^n(\xi)_1\in [\ve,1-\ve] \textup{ mod }1$ or $\alpha^n(\xi)_2\in  [\ve,1-\ve] \textup{ mod }1.$ Again it now easily follows from \cite[Theorem 25]{Hadwin87} that $G$ has a quotient that is not finite.
\end{proof}

\section*{Acknowledgements} I would like to thank Christopher Lee for pointing out Lemma \ref{lem:topfree} to me.  I would also like to thank the operators of the website mathoverflow.net and Terry Tao for answering my question on mathoverflow (see \cite{Tao13}) proving Lemma \ref{lem:expgrowth}.

\end{document}